\documentclass{amsart}

\usepackage{amsmath}
\usepackage{amssymb}
\usepackage{amsthm}
\usepackage{url}
\usepackage{enumitem}
\usepackage{MnSymbol}

\usepackage{hyperref}

\newtheorem{theorem}{Theorem}[section]
\newtheorem{lemma}[theorem]{Lemma}
\newtheorem{corollary}[theorem]{Corollary}

\theoremstyle{definition}
\newtheorem{definition}[theorem]{Definition}

\newtheorem{question}[theorem]{Question}

\theoremstyle{remark}

\newcommand{\mc}[1]{\mathcal{#1}}

\newcommand{\la}[0]{\langle}
\newcommand{\ra}[0]{\rangle}

\newcommand{\infi}[0]{\mathtt{in}}
\newcommand{\comp}[0]{\mathtt{c}}

\newcommand{\Th}[0]{\Theory_{\infty}}

\def\Do{\mathcal Dom}

\DeclareMathOperator{\Theory}{Th}
\DeclareMathOperator{\SR}{SR}
\DeclareMathOperator{\Mod}{Mod}

\begin{document}

\title[A first-order theory of Ulm type]{A first-order theory of Ulm type}

\author{Matthew Harrison-Trainor}
\address{Group in Logic and the Methodology of Science\\
University of California, Berkeley\\
 USA}
\email{matthew.h-t@berkeley.edu}
\urladdr{\href{http://www.math.berkeley.edu/~mattht/index.html}{www.math.berkeley.edu/$\sim$mattht}}

\begin{abstract}
The class of abelian $p$-groups are an example of some very interesting phenomena in computable structure theory. We will give an elementary first-order theory $T_p$ whose models are each bi-interpretable with the disjoint union of an abelian $p$-group and a pure set (and so that every abelian $p$-group is bi-interpretable with a model of $T_p$) using computable infinitary formulas. This answers a question of Knight by giving an example of an elementary first-order theory of ``Ulm type'': Any two models, low for $\omega_1^{CK}$, and with the same computable infinitary theory, are isomorphic. It also gives a new example of an elementary first-order theory whose isomorphism problem is $\mathbf{\Sigma}^1_1$-complete but not Borel complete.
\end{abstract}

\maketitle

\section{Introduction}

The class of abelian $p$-groups is a well-studied example in computable structure theory. A simple compactness argument shows that abelian $p$-groups are not axiomatizable by an elementary first-order theory, but they are definable by the conjunction of the axioms for abelian $p$-groups (which are first-order $\forall \exists$ sentences) and the infinitary $\Pi^0_2$ sentence which says that every element is torsion of order some power of $p$.

Abelian $p$-groups are classifiable by their Ulm sequences \cite{Ulm33}. Due to this classification, abelian $p$-groups are examples of some very interesting phenomena in computable structure theory and descriptive set theory. We will define a theory $T_p$ whose models behave like the class of abelian $p$-groups, giving a first-order example of these phenomena. In particular, Theorem \ref{thm:knight} below answers a question of Knight.

\subsection{Infinitary Formulas}

The infinitary logic $\mc{L}_{\omega_1 \omega}$ is the logic which allows countably infinite conjunctions and disjunctions but only finite quantification. If the conjunctions and disjunctions of a formula $\varphi$ are all over computable sets of indices for formulas, then we say that $\varphi$ is computable. We use $\Sigma_\alpha^\infi$ and $\Pi_\alpha^\infi$ to denote the classes of all infinitary $\Sigma_\alpha$ and $\Pi_\alpha$ formulas respectively. We will also use $\Sigma_\alpha^\comp$ and $\Pi_\alpha^\comp$ to denote the classes of computable $\Sigma_\alpha$ and $\Pi_\alpha$ formulas, where $\alpha < \omega_1^{CK}$ the least non-computable ordinal. See Chapter 6 of \cite{AshKnight00} for a more complete description of computable formulas.

\subsection{Bi-Interpretability}

One way in which we will see that the models of $T_p$ are essentially the same as abelian $p$-group is using bi-interpretations using infinitary formulas \cite{MonICM,HTMelnikovMillerMontalban,HTMillerMontalban}. A structure $\mc{A}$ is infinitary interpretable in a structure $\mc{B}$ if there is an interpretation of $\mc{A}$ in $\mc{B}$ where the domain of the interpretation is allowed to be a subset of $\mc{B}^{< \omega}$ and where all of the sets in the interpretation are definable using infinitary formulas. This differs from the classical notion of interpretation, as in model theory [Mar02, Definition 1.3.9], where the domain is required to be a subset of $\mc{B}^n$ for some $n$, and the sets in the interpretation are first-order definable.

\begin{definition}\label{def: eff int}
We say that a structure ${\mc{A}} = (A; P_0^{\mc{A}},P_1^{\mc{A}},...)$ (where $P_i^{\mc{A}}\subseteq A^{a(i)}$) is {\em infinitary interpretable} in $\mc{B}$ if there exists a sequence of relations $(\Do_{\mc{A}}^{\mc{B}}, \sim, R_0, R_1,...)$, definable using infinitary formulas (in the language of $\mc{B}$, without parameters), such that
\begin{enumerate}
	\item $\Do^{\mc{B}}_{\mc{A}}\subseteq \mc{B}^{<\omega}$,
	\item $\sim$  is an equivalence relation on $\Do^{\mc{B}}_{{\mc{A}}}$, 
	\item $R_i\subseteq (B^{<\omega})^{a(i)}$ is closed under $\sim$ within $\Do_{\mc{A}}^{\mc{B}}$,
\end{enumerate}
and there exists a function $f^{\mc{B}}_{\mc{A}} \colon \Do^{\mc{B}}_{{\mc{A}}} \to {\mc{A}}$ which induces an isomorphism: 
\[
(\Do^{\mc{B}}_{\mc{A}}/\sim; R_0/\sim,R_1/\sim,...) \cong (A; P_0^{\mc{A}},P_1^{\mc{A}},...),
\]
where $R_i/ \sim$ stands for the $\sim$-collapse of $R_i$.
\end{definition}

Two structures $\mc{A}$ and $\mc{B}$ are infinitary bi-interpretable if they are each effectively interpretable in the other, and moreover, the composition of the interpretations---i.e., the isomorphisms which map $\mc{A}$ to the copy of $\mc{A}$ inside the copy of $\mc{B}$ inside $\mc{A}$, and $\mc{B}$ to the copy of $\mc{B}$ inside the copy of $\mc{A}$ inside $\mc{B}$---are definable.

\begin{definition}
\label{defn:eff-biinterpretable}
Two structures $\mc{A}$ and $\mc{B}$ are {\em infinitary bi-interpretable} if there are infinitary
interpretations of each structure in the other as in Definition \ref{def: eff int} such that the compositions
\[
f^{\mc{A}}_{\mc{B}} \circ \tilde{f}^{\mc{B}}_{\mc{A}}\colon \Do_{\mc{B}}^{(\Do_{\mc{A}}^{\mc{B}})} \to {\mc{B}} 
\quad\mbox{ and }\quad
f^{\mc{B}}_{\mc{A}} \circ \tilde{f}^{\mc{A}}_{\mc{B}}\colon \Do_{\mc{A}}^{(\Do_{\mc{B}}^{\mc{A}})} \to {\mc{A}} 
\]
are  definable in ${\mc{B}}$ and ${\mc{A}}$ respectively.
(Here, we have  $\Do_{\mc{B}}^{(\Do_{\mc{A}}^{\mc{B}})}\subseteq (\Do_{\mc{A}}^{\mc{B}})^{<\omega}$, and $\tilde{f}^{\mc{B}}_{\mc{A}}\colon (\Do_{\mc{A}}^{\mc{B}})^{<\omega}\to {\mc{A}}^{<\omega}$ is the obvious extension of $f^{\mc{B}}_{\mc{A}}\colon \Do_{\mc{A}}^{\mc{B}}\to {\mc{A}}$ mapping $\Do_{\mc{B}}^{(\Do_{\mc{A}}^{\mc{B}})}$ to $\Do_{\mc{B}}^{\mc{A}}$.)
\end{definition}

If we ask that the sets and relations in the interpretation (or bi-interpretation) be (uniformly) relatively intrinsically computable, i.e., definable by both a $\Sigma^{\comp}_1$ formula and a $\Pi^{\comp}_1$ formula, then we say that the interpretation (or bi-interpretation) is effective. Any two structures which are effectively bi-interpretable have all of the same computability-theoretic properties; for example, they have the same degree spectra and the same Scott rank. See \cite[Lemma 5.3]{MonICM}.

Here, we will use interpretations which use (lightface) $\Delta^{\comp}_2$ formulas. It is no longer true that any two structures which are $\Delta^{\comp}_2$-bi-interpretable have all of the same computability-theoretic properties, but it is true, for example, that any two such structures either both have computable, or both have non-computable, Scott rank.

\begin{theorem}\label{thm:bi-int}
Each abelian $p$-group is effectively bi-interpretable with a model of $T_p$.
Each model of $T_p$ is $\Delta^{\comp}_2$-bi-interpretable with the disjoint union of an abelian $p$-group and a pure set.
\end{theorem}

\noindent This theorem will follow from the constructions in Sections \ref{sec:m} and \ref{sec:g}. Given a model $\mc{M}$ of $T_p$, $\mc{M}$ is bi-interpretable with an abelian $p$-group $G$ and a pure set. The domain of the copy of $G$ inside of $\mc{M}$ is definable by a $\Sigma^{\comp}_1$ formula but not by a $\Pi^{\comp}_1$ formula. This is the only part of the bi-interpretation which is not effective.

\subsection{Classification via Ulm Sequences}

Let $G$ be an abelian group. For any ordinal $\alpha$, we can define $p^\alpha G$ by transfinite induction:
\begin{itemize}
	\item $p^0 G = G$;
	\item $p^{\alpha + 1} G = p (p^\alpha G)$;
	\item $p^\beta G = \bigcap_{\alpha < \beta} p^\alpha G$ if $\beta$ is a limit ordinal.
\end{itemize}
These subgroups $p^\alpha G$ form a filtration of $G$. This filtration stabilizes, and we call the smallest ordinal $\alpha$ such that $p^\alpha G = p^{\alpha + 1} G$ the length of $G$. We call the intersection $p^\infty G$ of these subgroups, which is a $p$-divisible group, the $p$-divisible part of $G$. Any countable $p$-divisible group is isomorphic to some direct product of the Pr\"ufer group
\[ \mathbb{Z}(p^\infty) = \mathbb{Z}[1/p,1/p^2,1/p^3,\ldots] / \mathbb{Z}.\]

Denote by $G[p]$ the subgroup of $G$ consisting of the $p$-torsion elements. The $\alpha$th Ulm invariant $u_\alpha(G)$ of $G$ is the dimension of the quotient
\[ \big( p^\alpha G \big)[p] \; /  \; \big(p^{\alpha + 1} G \big) [p]\]
as a vector space over $\mathbb{Z} / p\mathbb{Z}$.

\begin{theorem}[Ulm's Theorem, see \cite{Fuchs}]
Let $G$ and $H$ be countable abelian p-groups such that for every ordinal $\alpha$ their $\alpha$th Ulm invariants are equal, and the $p$-divisible parts of $G$ and $H$ are isomorphic. Then $G$ and $H$ are isomorphic.
\end{theorem}

\subsection{Scott Rank and Computable Infinitary Theories}

Scott \cite{Scott65} showed that if $\mc{M}$ is a countable structure, then there is a sentence $\varphi$ of $\mc{L}_{\omega_1 \omega}$ such that $\mc{M}$ is, up to isomorphism, the only countable model of $\varphi$. We call such a sentence a Scott sentence for $\mc{M}$.
There are many different definitions \cite[Sections 6.6 and 6.7]{AshKnight00} of the Scott rank of $\mc{M}$, which differ only slightly in the ranks they assign. The one we will use, which comes from \cite{Montalban15}, defines the Scott rank of $\mc{A}$ to be the least ordinal $\alpha$ such that $\mc{A}$ has a $\Pi^{\infi}_{\alpha+1}$ Scott sentence. We denote the Scott rank of a structure $\mc{A}$ by $\SR(\mc{A})$. It is always the case that $\SR(\mc{A}) \leq \omega_1^{\mc{A}} + 1$ \cite{Nadel74}. We could just as easily use any of the other definitions of Scott rank; for all of these definitions, given a computable structure $\mc{A}$:
\begin{enumerate}
\item  $\mathcal{A}$ has computable Scott rank if and only if there is a computable ordinal $\alpha$ such that for all tuples $\bar{a}$ in $\mathcal{A}$, the orbit of $\bar{a}$ is defined by a computable $\Sigma_\alpha$ formula.  
\item  $\mathcal{A}$ has Scott rank $\omega_1^{CK}$ if and only if for each tuple $\bar{a}$, the orbit is defined by a computable infinitary formula, but for each computable ordinal $\alpha$, there is a tuple $\bar{a}$ whose orbit is not defined by a computable $\Sigma_\alpha$ formula.  
\item  $\mathcal{A}$ has Scott rank $\omega_1^{CK}+1$ if and only if there is a tuple $\bar{a}$ whose orbit is not defined by a computable infinitary formula.  
\end{enumerate}

Given a structure $\mc{M}$, define the computable infinitary theory of $\mc{M}$, $\Th(\mc{M})$, to be collection of computable $\mc{L}_{\omega_1 \omega}$ sentences true of $\mc{M}$. We can ask, for a given structure $\mc{M}$, whether $\Th(\mc{M})$ is $\aleph_0$-categorical, or whether there are other countable models of $\Th(\mc{M})$. For $\mc{M}$ a hyperarithmetic structure:
\begin{enumerate}
	\item If $\SR(\mc{M}) < \omega_1^{CK}$, then $\Th(\mc{M})$ is $\aleph_0$-categorical. Indeed, $\mc{M}$ has a computable Scott sentence \cite{Nadel74}.
		\item If $\SR(\mc{M}) = \omega_1^{CK}$, then $\Th(\mc{M})$ may or may not be $\aleph_0$-categorical \cite{HTIgusaKnight}.
	\item If $\SR(\mc{M}) = \omega_1^{CK} + 1$, then $\Th(\mc{M})$ is not $\aleph_0$-categorical  as $\mc{M}$ has a non-principal type which may be omitted.
\end{enumerate}
In the case of abelian $p$-groups, we can say something even when we replace the assumption that $\mc{M}$ is hyperarithmetic with the assumption that $\omega_1^G = \omega_1^{CK}$.

\begin{definition}[Definition 6 of \cite{FokinaKnightMelnikovQuinnSafranski11}]
A class of countable structures has \textit{Ulm type} if for any two structures $\mc{A}$ and $\mc{B}$ in the class, if $\omega_1^{\mc{A}} = \omega_1^{\mc{B}} = \omega_1^{CK}$ and $\Th(\mc{A}) = \Th(\mc{B})$, then $\mc{A}$ and $\mc{B}$ are isomorphic.
\end{definition}

\noindent It is well-known that abelian $p$-groups are of Ulm type; however, we do not know of a good reference with a complete proof, so we will give one in Section \ref{sec:abelian}. We also note that there are indeed non-hyperarithmetic abelian $p$-groups $G$ with $\SR(G) < \omega_1^{CK}$.

Knight asked whether there was a (non-trivial) first-order theory of Ulm type. By a non-trivial example, we mean that the elementary first-order theory should have non-hyperarithmetic models which are low for $\omega_1^{CK}$. Our theory $T_p$ is such an example.

\begin{theorem}\label{thm:knight}
The models of $T_p$ are of Ulm type. Moreover, given $\mc{M} \models T_p$ with $\omega_1^{CK} = \omega_1^{\mc{M}}$ and $\SR(\mc{M}) < \omega_1^{CK} = \omega_1^{\mc{M}}$, $\Th(\mc{M})$ is $\aleph_0$-categorical.
\end{theorem}
\begin{proof}
Let $\mc{M}$ be a model of $T_p$. Now $\mc{M}$ is bi-interpretable, using computable infinitary formulas, with the disjoint union of an abelian $p$-group $G$ and a pure set. Thus $\mc{M}$ inherits these properties from $G$ (see Theorem \ref{thm:abelian}).
\end{proof}

\noindent Of course, there will be non-hyperarithmetic models of $T_p$ with Scott rank below $\omega_1^{CK}$.

\subsection{Borel Incompleteness}

In their influential paper \cite{FriedmanStanley89}, Friedman and Stanley introduced Borel reductions between invariant Borel classes of structures with universe $\omega$ in a countable language. Such classes are of the form $\Mod(\varphi)$, the set of models of $\varphi$ with universe $\omega$, for some $\varphi \in \mc{L}_{\omega_1 \omega}$. A Borel reduction from $\Mod(\varphi)$ to $\Mod(\psi)$ is a Borel map $\Phi \colon \Mod(\varphi) \to \Mod(\psi)$ such that
\[ \mc{M} \cong \mc{N} \Longleftrightarrow \Phi(\mc{M}) \cong \Phi(\mc{N}).\]
If such a Borel reduction exists, we say that $\Mod(\varphi)$ is Borel reducible to $\Mod(\psi)$ and write $\varphi \leq_B \psi$. If $\varphi \leq_B \psi$ and $\psi \leq_B \varphi$, then we say that $\Mod(\varphi)$ and $\Mod(\psi)$ are Borel equivalent and write $\varphi \equiv_B \psi$. Friedman and Stanley showed that graphs, fields, linear orders, trees, and groups are all Borel equivalent, and form a maximal class under Borel reduction.

If $\Mod(\varphi)$ is Borel complete, then the isomorphism relation on $\Mod(\varphi) \times \Mod(\varphi)$ is $\mathbf{\Sigma}^1_1$-complete. The converse is not true, and the most well-known example is abelian $p$-groups, whose isomorphism relation is $\mathbf{\Sigma}^1_1$-complete but not Borel complete. Until very recently, they were one of the few such examples, and there were no known examples of elementary first-order theories with similar properties. Recently, Laskowski, Rast, and Ulrich \cite{LaskowskiRastUlrich} gave an example of a first-order theory which is not Borel complete, but whose isomorphism relation is not Borel. Our theory $T_p$ is another such example.

\begin{theorem}\label{thm:borel}
The class of models of $T_p$ is Borel equivalent to abelian $p$-groups.
\end{theorem}

Because abelian $p$-groups are not Borel complete, but their isomorphism relation is $\mathbf{\Sigma^1_1}$-complete, we get:

\begin{corollary}
The class of models of $T_p$ is not Borel complete but the isomorphism relation is $\mathbf{\Sigma^1_1}$-complete.
\end{corollary}

Theorem \ref{thm:borel} is a specific instance of the following general question asked by Friedman:

\begin{question}
Is it true that for every $\mc{L}_{\omega_1 \omega}$ sentence there is a Borel equivalent first-order theory?
\end{question}

\section{\texorpdfstring{Abelian $p$-groups are of Ulm type}{Abelian p-groups are of Ulm type}}\label{sec:abelian}

In this section we will describe a proof of the following well-known theorem, which shows that abelian $p$-groups are of Ulm type.

\begin{theorem}\label{thm:abelian}
Let $G$ be an abelian $p$-group with $\omega_1^{CK} = \omega_1^G$. Then:
\begin{enumerate}
	\item $G$ is the only countable model of $\Th(G)$ with $\omega_1^G = \omega_1^{CK}$, and
	\item if $\SR(G) < \omega_1^{CK} = \omega_1^G$, then $\Th(G)$ is $\aleph_0$-categorical.
\end{enumerate}
\end{theorem}

The proof of Theorem \ref{thm:abelian} consists essentially of expressing the Ulm invariants via computable infinitary formulas.

\begin{definition}
Let $G$ be an abelian $p$-group. For each ordinal $\alpha < \omega_1^{CK}$, there is a computable infinitary sentence $\psi_{\alpha}(x)$ which defines $p^\alpha G$ inside of $G$:
\begin{itemize}
	\item $\psi_0(x)$ is just $x = x$; 
	\item $\psi_{\alpha + 1}(x)$ is $(\exists y)[\psi_\alpha(y) \wedge py = x]$;
	\item $\psi_{\beta}(x)$ is $\bigdoublewedge_{\alpha < \beta} \psi_{\alpha}(x)$ for limit ordinals $\beta$.
\end{itemize}
\end{definition}

\begin{definition}
For each ordinal $\alpha < \omega_1^{CK}$ and $n \in \omega \cup \{\omega\}$, there is a computable infinitary sentence $\varphi_{\alpha,n}$ such that, for $G$ an abelian $p$-group,
\[ G \models \varphi_{\alpha,n} \Leftrightarrow u_\alpha(G) = n. \] 
For $n \in \omega$, define $\varphi_{\alpha,\geq n}$ to say that there are $x_1,\ldots,x_n$ such that:
\begin{itemize}
	\item $\psi_\alpha(x_1) \wedge \cdots \wedge \psi_\alpha(x_n)$,
	\item $px_1 = \cdots = px_n = 0$, and
	\item for all $c_1,\ldots,c_n \in \mathbb{Z} / p\mathbb{Z}$ not all zero, $\neg \psi_{\alpha + 1}(c_1 x_1 + \cdots + c_n x_n)$.
\end{itemize}
Then for $n \in \omega$, $\varphi_{\alpha,n}$ is $\varphi_{\alpha,\geq n} \wedge \neg \varphi_{\alpha,\geq n+1}$, and $\varphi_{\alpha,\omega}$ is $\bigdoublewedge_{n \in \omega} \varphi_{\alpha,\geq n}$.
\end{definition}

\begin{lemma}[Theorem 8.17 of \cite{AshKnight00}]
Let $G$ be an abelian $p$-group. Then:
\begin{enumerate}
	\item the length of $G$ is at most $\omega_1^{G}$, and
	\item if $G$ has length $\omega_1^G$ then $G$ is not reduced (in fact, its $p$-divisible part has infinite rank) and $\SR(G) = \omega_1^{G} + 1$.
\end{enumerate}
\end{lemma}

\noindent We are now ready to give the proof of Theorem \ref{thm:abelian}.

\begin{proof}[Proof of Theorem \ref{thm:abelian}]
Since $\omega_1^{CK} = \omega_1^G$, $G$ has length at most $\omega_1^{CK}$. Note that $\Th(G)$ contains the sentences $\varphi_{\alpha,u_\alpha(G)}$ for $\alpha < \omega_1^{CK}$. Thus any model of $\Th(G)$ has the same Ulm invariants as $G$, for ordinals below $\omega_1^{CK}$.

If $\SR(G) < \omega_1^{CK}$, let $\lambda$ be the length of $G$. Then $\Th(G)$ includes the computable formula $(\forall x)[\psi_{\lambda}(x) \leftrightarrow \psi_{\lambda + 1}(x)]$, so that any countable model of $\Th(G)$ has length at most $\lambda$. Note that in such a model, $\psi_{\lambda}$ defines the $p$-divisible part. Let $n \in \omega \cup \{\omega\}$ be such that $p^\infty G$ is isomorphic to $\mathbb{Z}(p^\infty)^n$. Then, if $n \in \omega$, $\Th(G)$ contains the formula which says that there are $x_1,\ldots,x_n$ such that
\begin{itemize}
	\item $\psi_\lambda(x_1) \wedge \cdots \wedge \psi_\lambda(x_n)$,
	\item for all $c_1,\ldots,c_n < p$ not all zero and $k_1,\ldots,k_n \in \omega$,
	\[ \frac{c_1}{p^{k_1}} x_1 + \cdots +  \frac{c_n}{p^{k_n}} x_n \neq 0,\]
	\item for all $y$ with $\psi_{\lambda}(y)$, there are $c_1,\ldots,c_n < p$ and $k_1,\ldots,k_n \in \omega$ such that
	\[ y = \frac{c_1}{p^{k_1}} x_1 + \cdots +  \frac{c_n}{p^{k_n}} x_n.\]
\end{itemize}
If $n = \omega$, then $\Th(G)$ contains the formula which says that for each $m \in \omega$, there are $x_1,\ldots,x_m$ such that
\begin{itemize}
	\item $\psi_\lambda(x_1) \wedge \cdots \wedge \psi_\lambda(x_m)$, and
	\item for all $c_1,\ldots,c_m < p$ not all zero and $k_1,\ldots,k_m \in \omega$,
	\[ \frac{c_1}{p^{k_1}} x_1 + \cdots +  \frac{c_m}{p^{k_m}} x_m \neq 0.\]
\end{itemize}
Any countable model of $\Th(G)$ has $p$-divisible part isomorphic to $\mathbb{Z}(p^\infty)^n$. So any countable model of $\Th(G)$ has the same Ulm invariants and $p$-divisible part as $G$, and hence is isomorphic to $\Th(G)$. Hence $\Th(G)$ is $\aleph_0$-categorical. This gives (2), and (1) for the case where $\SR(G) < \omega_1^{CK}$.

If $\SR(G) = \omega_1^{CK} + 1$, let $H$ be any other countable model of $\Th(G)$ with $\omega_1^H = \omega_1^G = \omega_1^{CK}$. Thus $G$ and $H$ both have length $\omega_1^{CK}$ and their $p$-divisible parts have infinite rank. As remarked before, they have the same Ulm invariants, and so they must be isomorphic. This completes the proof of (1).
\end{proof}

\section{\texorpdfstring{The Theory $T_p$}{The Theory Tp}}\label{sec:m}

Fix a prime $p$. The language $\mc{L}_p$ of $T_p$ will consist of a constant $0$, unary relations $R_n$ for $n \in \omega$, and ternary relations $P_{\ell,m}^n$ for $\ell,m \in \omega$ and $n \leq \max(\ell,m)$. The following transformation of an abelian $p$-group into an $\mc{L}_p$-structure will illustrate the intended meaning of the symbols.

\begin{definition}
Let $G$ be an abelian $p$-group. Define $\mathfrak{M}(G)$ to be $\mc{L}_p$-structure obtained as follows, with the same domain as $G$, and the symbols of $\mc{L}_p$ interpreted as follows:
\begin{itemize}
	\item Set $0^{\mathfrak{M}(G)}$ to be the identity element of $G$.
	\item For each $n$, let $R_n^{\mathfrak{M}(G)}$ be the elements which are torsion of order $p^n$.
	\item For each $\ell,m \in \omega$ and $n \leq \max(\ell,m)$, and $x,y,z \in G$, set $P_{\ell,m}^{n,\mathfrak{M}(G)}(x,y,z)$ if and only if $x + y = z$, $x \in R_\ell^{\mathfrak{M}(G)}$, $y \in R_m^{\mathfrak{M}(G)}$, and $z \in R_n^{\mathfrak{M}(G)}$.
\end{itemize}
\end{definition}

\noindent One should think of such $\mc{L}_p$-structures as the canonical models of $T_p$. The theory $T_p$ will consist of following axiom schemata:
\begin{enumerate}[label=(A\arabic*)]
\item For all $\ell,m,n \in \omega$:
\[ (\forall x \forall y \forall z)\left[P_{\ell,m}^n(x,y,z) \rightarrow \left(R_\ell(x) \wedge R_m(x) \wedge R_n(z) \right) \right]. \]
\item \textit{($R_n$ contains the elements which are torsion of order $p^n$.)}
\[ (\forall x)[ R_0(x) \leftrightarrow x = 0].\]
and, for all $n \geq 1$:
\[(\forall x)\left[x \in R_n \leftrightarrow (\exists x_2 \cdots \exists x_{p-1})\left[P_{n,n}^{n}(x,x,x_2) \wedge P_{n,n}^{n}(x,x_2,x_3) \wedge \cdots \wedge P_{n,n}^{n-1}(x,x_{p-1},x_p) \right]\right].\]
\item \textit{($P$ defines a partial function.)} For all $\ell,m,n,n' \in \omega$:
\[ (\forall x \forall y \forall z \forall z')\left[ \big(P_{\ell,m}^n(x,y,z) \wedge P_{\ell,m}^{n'}(x,y,z') \big) \rightarrow z = z' \right]. \]
\item \textit{($P$ is total.)} For all $\ell,m \in \omega$:
\[ (\forall x \forall y)\left[ \big( R_\ell(x) \wedge R_m(y) \big) \rightarrow \bigvee_{n \leq \max(\ell,m)} (\exists z) P_{\ell,m}^n(x,y,z) \right]. \]
\item \textit{(Identity.)} For all $\ell \in \omega$:
\[ (\forall x)[R_\ell(x) \rightarrow \big[ P_{0,\ell}^{\ell}(0,x,x) \wedge P_{\ell,0}^\ell(x,0,x)]\big]. \]
\item \textit{(Inverses.)} For all $\ell \in \omega$:
\[ (\forall x)(\exists y)\big[R_\ell(x) \rightarrow [P_{\ell,\ell}^0(x,y,0) \wedge P_{\ell,\ell}^0(y,x,0)]\big].  \]
\item \textit{(Associativity.)} For all $\ell,m,n \in \omega$:
\begin{equation*}\begin{split} &(\forall x \forall y \forall z)\bigg[ \Big[ R_\ell(x) \wedge R_m(y) \wedge R_n(z) \Big] \longrightarrow \\
&\bigvee_{\substack{r \leq \max(\ell,m) \\ s \leq \max(m,n) \\ t \leq \max(r,n),\max(\ell,s)}} (\exists u \exists v \exists w) \Big[ P_{\ell,m}^r(x,y,u) \wedge P_{r,n}^{t}(u,z,w) \wedge P_{m,n}^{s}(y,z,v) \wedge P_{\ell,s}^{t} (x,v,w) \Big] \bigg].
\end{split}\end{equation*}
\item \textit{(Abelian.)} For all $\ell,m \in \omega$ and $n \leq \max(\ell,m)$:
\[ (\forall x \forall y \forall z)\big[[R_\ell(x) \wedge R_m(y) \wedge R_n(z) \wedge P_{\ell,m}^n(x,y,z)] \rightarrow P_{m,\ell}^n(y,x,z)\big]. \]
\end{enumerate}

We must now check that the definition of $T_p$ works as desired, that is, that if $G$ is an abelian $p$-group, then $\mathfrak{M}(G)$ is a model of $T_p$.

\begin{lemma}
If $G$ is an abelian $p$-group, then $\mathfrak{M}(G) \models T_p$.
\end{lemma}
\begin{proof}
We must check that each instance of the axiom schemata of $T_p$ holds in $\mathfrak{M}(G)$.
\begin{enumerate}[label=(A\arabic*)]
	\item Suppose that $x$, $y$, and $z$ are elements of $G$ with $P_{m,\ell}^{n,\mathfrak{M}(G)}(x,y,z)$. Then, by definition, $x + y = z$, $x \in R_\ell^\mathfrak{M}(G)$, $y \in R_m^{\mathfrak{M}(G)}$, and $z \in R_n^{\mathfrak{M}(G)}$.
	\item $R_0^{\mathfrak{M}(G)}$ contains the elements of $G$ which are torsion of order $p^0 = 1$, so $R_0$ contains just the identity. For each $n > 0$, $R_n^{\mathfrak{M}(G)}$ contains the elements of order $p^n$. An element $x$ has order $p^n$ if and only if $px$ has order $p^{n-1}$. It remains only to note that if $x$ has order $p^n$, then $x,2x,3x,\ldots,(p-1)x$ all have order $p^n$ as well. The existential quantifier is witnessed by $x_2 = 2x$, $x_3 = 3x$, and so on.
	\item If, for some $x$, $y$, $z$, and $z'$, $P^{n,\mathfrak{M}(G)}_{\ell,m}(x,y,z)$ and $P_{\ell,m}^{n',\mathfrak{M}(G)}(x,y,z')$, then $x + y = z$ and $x + y = z'$, so that $z = z'$.
	\item Given $x$ and $y$ in $G$ which are of order $p^m$ and $p^\ell$ respectively, $x + y$ is of order $p^n$ for some $n \leq \max(m,\ell)$, and so we have $P_{m,\ell}^{n,\mathfrak{M}(G)}(x,y,x+y)$.
	\item If $x \in G$ is of order $p^\ell$, then $x + 0 = 0 + x = x$ and so we have $P_{\ell,0}^{\ell,\mathfrak{M}(G)}(x,0,x)$.
	\item If $x \in G$ is of order $p^\ell$, then $-x$ is also of order $p^\ell$, and $x + (-x) = 0 = (-x) + x$. So we have $P_{\ell,\ell}^{0,\mathfrak{M}(G)}(x,-x,0)$.
	\item Given $x,y,z \in G$ of order $p^\ell$, $p^m$, and $p^n$ respectively, there are $r \leq \max(\ell,m)$ and $s \leq \max(m,n)$ such that $x + y$ and $y + z$ are of order $p^r$ and $p^s$ respectively. Then there is $t$ such that $x+y+z$ is of order $p^t$; $t \leq \max(r,n)$ and $t \leq \max(\ell,s)$.
	\item Given $x,y,z \in G$ of order $p^\ell$, $p^m$, and $p^n$ respectively, $n \leq \max(\ell,m)$, and with $x + y = z$, we have $y + x = z$ as $G$ is abelian.
\end{enumerate}
Thus we have shown that $\mathfrak{M}(G)$ is a model of $T_p$.
\end{proof}

Note that $G$ and $\mathfrak{M}(G)$ are effectively bi-interpretable, proving one half of Theorem \ref{thm:bi-int}.

\section{\texorpdfstring{From a model of $T_p$ to an abelian $p$-group}{From a model of Tp to an abelian p-group}}\label{sec:g}

Given an abelian $p$-group $G$, we have already described how to turn $G$ into a model of $T_p$. In this section we will do the reverse by turning a model of $T_p$ into an abelian $p$-group.

\begin{definition}
Let $\mc{M}$ be a model of $T_p$. Define $\mathfrak{G}(\mc{M})$ to be the group obtained as follows.
\begin{itemize}
	\item The domain of $\mathfrak{G}(\mc{M})$ will be the subset of the domain of $\mc{M}$ given by $\bigcup_{n \in \omega} R_n^{\mc{M}}$.
	\item The identity element of $\mathfrak{G}(\mc{M})$ will be $0^{\mc{M}}$.
	\item We will have $x + y = z$ in $\mathfrak{G}(\mc{M})$ if and only if, for some $\ell$, $m$, and $n$, $P_{\ell,m}^{n,\mc{M}}(x,y,z)$.
\end{itemize}
\end{definition}

\noindent We will now check that $\mathfrak{G}(\mc{M})$ is always an abelian $p$-group.

\begin{lemma}\label{lem:mod-to-gp}
If $\mc{M}$ is a model of $T_p$, then $\mathfrak{G}(\mc{M})$ is an abelian $p$-group.
\end{lemma}
\begin{proof}
First we check that the operation $+$ on $\mathfrak{G}(\mc{M})$ defines a total function. Given $x,y \in \mathfrak{G}(\mc{M})$, choose $\ell$ and $m$ such that $x \in R^{\mc{M}}_\ell$ and $y \in R^{\mc{M}}_m$. Then by (A3) and (A4), there is a unique $n \leq \max(\ell,m)$ and a unique $z$ such that $P_{\ell,m}^{n,\mc{M}}(x,y,z)$. Thus $x + y = z$, and $z$ is unique.

Second, we check that $\mathfrak{G}(\mc{M})$ is in fact a group. To see that $0^{\mc{M}}$ is the identity, given $x \in \mathfrak{G}(\mc{M})$, there is $\ell$ such that $x \in R^{\mc{M}}_\ell$. By (A5), $P_{\ell,0}^{\ell,\mc{M}}(x,0^{\mc{M}},x)$ and $P_{0,\ell}^{\ell,\mc{M}}(0^{\mc{M}},x,0^{\mc{M}})$. Thus $x + 0^{\mc{M}} = 0^{\mc{M}} + x = x$, and $0^{\mc{M}}$ is the identity of $\mathfrak{G}(\mc{M})$. To see that $\mathfrak{G}(\mc{M})$ has inverses, given $x \in \mathfrak{G}(\mc{M})$, there is $\ell$ such that $x \in R^{\mc{M}}_\ell$, and by (A6) there is $y \in R^{\mc{M}}_\ell$ such that $P_{\ell,\ell}^{0,\mc{M}}(x,y,0^{\mc{M}})$ and $P_{\ell,\ell}^{0,\mc{M}}(y,x,0^{\mc{M}})$. Thus $x + y = y + x = 0^{\mc{M}}$, and so $y$ is the inverse of $x$. Finally, to see that $\mathfrak{G}(M)$ is associative, given $x,y,z \in \mathfrak{G}(\mc{M})$, there are $\ell$, $m$, and $n$ such that $x \in R^{\mc{M}}_\ell$, $y \in R^{\mc{M}}_m$, and $z \in R^{\mc{M}}_n$. Then by (A7) there are $r$, $s$, and $t$, and $u$, $v$, and $w$, such that $P_{\ell,m}^{r,\mc{M}}(x,y,u)$, $P_{r,n}^{t,\mc{M}}(u,z,w)$, $P_{m,n}^{s,\mc{M}}(y,z,v)$, and $P_{\ell,s}^{t,\mc{M}}(x,v,w)$. Thus $x + y = u$, $u + z = w$, $y + z = v$, and $x + v = w$. So $(x+ y) + z = x + (y + z)$. Thus $\mathfrak{G}(\mc{M})$ is associative. 

Third, to see that $\mathfrak{G}(\mc{M})$ is abelian, let $x,y \in \mathfrak{G}(\mc{M})$. There are $\ell$ and $m$ such that $x \in R^{\mc{M}}_\ell$ and $y \in R^{\mc{M}}_m$. Let $n \leq \max(\ell,m)$ be such that $z = x + y \in R_n^{\mc{M}}$. (Such an $n$ and $z$ exist by the arguments above that $+$ is total, via (A3) and (A4).) Then $P_{\ell,m}^{n,\mc{M}}(x,y,z)$, and so by (A8), $P_{m,\ell}^{n,\mc{M}}(y,x,z)$. Thus $y + x = z$ and so $\mathfrak{G}(\mc{M})$ is abelian.

Finally, we need to see that $\mathfrak{G}(\mc{M})$ is a $p$-group. We claim, by induction on $n \geq 0$, that $R^{\mc{M}}_n$ consists of the elements of $\mathfrak{G}(\mc{M})$ which are of order $p^n$. From this claim, it follows that $\mathfrak{G}(\mc{M})$ is a $p$-group. For $n = 0$, the claim follows directly from (A2). Given $n > 0$, suppose that $x \in R^{\mc{M}}_n$. Then the witnesses $x_2,x_3,\ldots,x_p$ to (A2) must be $2x,3x,\ldots,px$. Note that since $P_{n,n}^{n-1,\mc{M}}(x,(p-1)x,px)$, $px \in R^{\mc{M}}_{n-1}$. Thus $px$ is of order $p^{n-1}$, and so $x$ is of order $p^n$. On the other hand, if $x$ is of order $p^n$, then $px$ is of order $p^{n-1}$ and so $px \in R^{\mc{M}}_{n-1}$. Moreover, $x_2 = 2x, x_3 =3x,\ldots,x_{p-1} = (p-1)x$ are all of order $p^n$. So we have $P_{n,n}^{n,\mc{M}}(x,x,x_2),P_{n,n}^{n,\mc{M}}(x,x_2,x_3),\ldots,,P_{n,n}^{n-1,\mc{M}}(x,x_{p-1},x_p)$. By (A2), $x \in R^{\mc{M}}_n$. This completes the inductive proof.
\end{proof}

We now have two operations, one which turns an abelian $p$-group into a model of $T_p$, and another which turns a model of $T_p$ into an abelian $p$-group. These two operations are almost inverses to each other. If we begin with an abelian $p$-group, turn it into a model of $T_p$, and then that model into an abelian $p$-group, we will obtain the original group. However, if we start with a $\mc{M}$ model of $T_p$, turn it into an abelian $p$-group, and then turn that abelian $p$-group into a model of $T_p$, we may obtain a different model of $T_p$. The problem is that the of elements of $\mc{M}$ which are not in any of the sets $R_n^{\mc{M}}$ are discarded when we transform $\mc{M}$ into an abelian $p$-group. However, these elements form a pure set, and so the only pertinent information is their size.

\begin{definition}
Given a model $\mc{M}$ of $T_p$, the size of $\mc{M}$, $\# \mc{M} \in \omega \cup \{ \infty \}$, is the number of elements of $M$ not in any relation $R_n$. 
\end{definition}

\begin{lemma}
Given an abelian $p$-group $G$, $\mathfrak{G}(\mathfrak{M}(G)) = G$.
\end{lemma}
\begin{proof}
Since $\# \mathfrak{M}(G) = 0$, we see that $G$, $\mathfrak{M}(G)$, and $\mathfrak{G}(\mathfrak{M}(G))$ all have the same domain. The identity of $\mathfrak{G}(\mathfrak{M}(G))$ is $0^{\mathfrak{M}(G)}$ which is the identity of $G$. If $x + y = z$ in $G$, then, for some $\ell,m,n \in \omega$, we have $P_{\ell,m}^{n,\mathfrak{M}(G)}(x,y,z)$. Thus, in $\mathfrak{G}(\mathfrak{M}(G))$, we have $x + y = z$. So $\mathfrak{G}(\mathfrak{M}(G)) = G$.
\end{proof}

We make a simple extension to $\mathfrak{M}$ as follows.
\begin{definition}
Let $G$ be an abelian $p$-group and $m \in \omega \cup \{\infty\}$. Define $\mathfrak{M}(G,m)$ to be $\mc{L}_p$-structure with domain $G \cup \{ a_1,\ldots,a_m\}$ with the relations interpreted as in $\mathfrak{M}(G)$. Thus, no relations hold of any of the elements $a_1,\ldots,a_m$.
\end{definition}

\begin{lemma}\label{lem:get-m}
Given a model $\mc{M}$ of $T_p$, $\mathfrak{M}(G(\mc{M}),\#\mc{M}) \cong \mc{M}$.
\end{lemma}
\begin{proof}
We will show that if $\#\mc{M} = 0$, then $\mathfrak{M}(\mathfrak{G}(\mc{M})) = \mc{M}$. From this one can easily see that $\mathfrak{M}(G(\mc{M}),\#\mc{M}) \cong \mc{M}$ in general.

If $\#\mc{M} = 0$, then $\mc{M}$, $\mathfrak{G}(\mc{M})$, and $\mathfrak{M}(\mathfrak{G}(\mc{M}))$ all share the same domain. It is clear that $0^{\mc{M}} = 0^{\mathfrak{G}(\mc{M})} = 0^{\mathfrak{M}(\mathfrak{G}(\mc{M}))}$. From the proof of Lemma \ref{lem:mod-to-gp}, we see that for each $n$, $R_n^{\mc{M}}$ defines the set of elements of $\mathfrak{G}(\mc{M})$ which are torsion of order $p^n$, and so $R_n^{\mc{M}} = R_n^{\mathfrak{M}(\mathfrak{G}(\mc{M}))}$. Given $\ell,m \in \omega$ and $n \leq \max(\ell,m)$, and $x$, $y$, and $z$ elements of the shared domain, we have $P_{\ell,m}^{n,\mc{M}}(x,y,z)$ if and only if
\[ \text{$x + y = z$ in $\mathfrak{G}(\mc{M})$ and $x \in R_\ell^{\mc{M}}$, $y \in R_m^{\mc{M}}$, and $z \in R_n^{\mc{M}}$}.\]
Since $R_i^{\mc{M}} = R_i^{\mathfrak{M}(\mathfrak{G}(\mc{M}))}$ for each $i$, this is the case if and only if $P_{\ell,m}^{n,\mathfrak{M}(\mathfrak{G}(\mc{M}))}(x,y,z)$. Thus we have shown that $\mathfrak{M}(\mathfrak{G}(\mc{M})) = \mc{M}$.
\end{proof}

Note that $\mc{M}$ and the disjoint union of $\mathfrak{G}(\mc{M})$ with a pure set of size $\# \mc{M}$ are bi-interpretable, using computable infinitary formulas, completing the proof of Theorem \ref{thm:bi-int}.

\section{Borel Equivalence}

In this section we will prove Theorem \ref{thm:borel} by showing that the class of models of $T_p$ and the class of abelian $p$-groups are Borel equivalent. $G \mapsto \mathfrak{G}(\mathfrak{M}(G)) = \mathfrak{G}(\mathfrak{M}(G,0))$ is a Borel reduction from isomorphism on abelian $p$-groups to isomorphism on models of $T_p$. However, $\mc{M} \mapsto \mathfrak{G}(\mc{M})$ is not a Borel reduction in the other direction, because two non-isomorphic models of $T_p$ might be mapped to isomorphic groups. We need to find a way to turn $\mathfrak{G}(\mc{M})$ and $\# \mc{M}$ into an abelian $p$-group $\mathfrak{H}(\mathfrak{G}(\mc{M}),\# \mc{M})$, so that $\mc{M}$ and $\# \mc{M}$ can be recovered from $\mathfrak{H}(\mathfrak{G}(\mc{M}),\#\mc{M})$.

We will define $\mathfrak{H}(G,m)$ for any abelian $p$-group $H$ and $m \in \omega \cup \{\infty\}$. It is helpful to think about what this reduction will do to the Ulm invariants: The first Ulm invariant of $\mathfrak{H}(G,m)$ will be $m$, and for each $\alpha$, then $1+\alpha$th Ulm invariant of $\mathfrak{H}(G,m)$ will be the same as the $\alpha$th Ulm invariant of $G$.

\begin{definition}
Given an abelian $p$-group $G$, and $m \in \omega \cup \{\infty\}$, define an abelian $p$-group $\mathfrak{H}(G,m)$ as follows. Let $\hat{\mc{B}}$ be a basis for the $\mathbb{Z}_p$-vector space $G / {pG}$. Let $\mc{B} \subseteq G$ be a set of representatives for $\hat{\mc{B}}$. Let $G^*$ be the abelian group $\la G,a_b : b \in \mc{B} \mid pa_b = b \ra$. Then define $\mathfrak{H}(G,m) = G^* \oplus (\mathbb{Z}_p)^{m}$.
\end{definition}

To make this Borel, we can take $\mc{B}$ to be the lexicographically first set of representatives for a basis. It will follow from Lemma \ref{lem:isotype} that the isomorphism type of $\mathfrak{H}(G,m)$ does not depend on these choices. First, we require a couple of lemmas.

\begin{lemma}\label{lem:uniq-g}
Each element of $G$ can be written uniquely as a (finite) linear combination $h + \sum_{b \in \mc{B}} x_b b$ where $h \in pG$ and each $x_b < p$.\end{lemma}
\begin{proof}
Given $g \in G$, let $\hat{g}$ be the image of $g$ in $G / pG$. Then, since $\hat{\mathcal{B}}$ is a basis for $G / p G$, we can write
\[ \hat{g} =  \sum_{b \in \mc{B}} x_b \hat{b} \]
with $x_b < p$, where $\hat{b}$ is the image of $b$ in $G / pG$.
Thus setting
\[ h = g - \sum_{b \in \mc{B}} x_b b \in pG \]
we get a representation of $g$ as in the statement of the theorem.

To see that this representation is unique, suppose that
\[ h + \sum_{b \in \mc{B}} x_b b = h' + \sum_{b \in \mc{B}} y_b b.\]
Then, modulo $pG$,
\[ \sum_{b \in \mc{B}} x_b \hat{b} = \sum_{b \in \mc{B}} y_b \hat{b}.\]
Since $\hat{\mc{B}}$ is a basis, $x_b = y_b$ for each $b \in \mc{B}$.
Then we get that $h = h'$ and the two representations are the same.
\end{proof}

\begin{lemma}\label{lem:unique-rep}
Each element of $G^*$ can be written uniquely in the form $h + \sum_{b \in \mc{B}} x_b a_{b}$ where $h \in G$ and each $x_b < p$.\end{lemma}
\begin{proof}
It is clear that each element of $G^*$ can be written in such a way. If
\[ h + \sum_{b \in \mc{B}} x_b a_{b} = h' + \sum_{b \in \mc{B}} y_b a_{b} \]
then, in $G$,
\[ ph + \sum_{b \in \mc{B}} x_b b = ph' + \sum_{b \in \mc{B}} y_b b.\]
This representation is unique, so $x_b = y_b$ for each $b \in \mc{B}$, and so $h = h'$.
\end{proof}

\begin{lemma}\label{lem:isotype}
The isomorphism type of $\mathfrak{H}(G,m)$ depends only on the isomorphism type of $G$, and not on the choice of $\mc{B}$.
\end{lemma}
\begin{proof}
It suffices to show that if $\mc{C}$ is another choice of representatives for a basis of $G / {pG}$, then $G^*_{\mc{B}} = G^*_{\mc{C}}$, where the former is constructed using $\mc{B}$, and the later is constructed using $\mc{C}$. Let $f \colon \mc{B} \to \mc{C}$ be an bijection.

Given $g \in G^*_{\mc{B}}$, write $g = g' + \sum_{b \in \mc{B}} x_b a_{b}$ with $g' \in G$ and $0 \leq x_b < p$. This representation of $g$ is unique by Lemma \ref{lem:unique-rep}.
Define $\varphi(g) = g' + \sum_{b \in \mc{B}} x_b a_{f(b)}$. It is not hard to check that $\varphi$ is a homomorphism. The inverse of $\varphi$ is the map $\psi$ which is defined by $\psi(h) = h' + \sum_{c \in \mc{C}} y_c a_{f^{-1}(c)}$ where $h = h' + \sum_{c \in \mc{C}} y_c a_c$.
\end{proof}

The next two lemmas will be used to show that if $G$ is not isomorphic to $G'$, or if $m$ is not equal to $m'$, then $\mathfrak{H}(G,m)$ will not be isomorphic to $\mathfrak{H}(G',m')$. 

\begin{lemma}
$G = pG^*$.
\end{lemma}
\begin{proof}
Each element of $G$ can be written as $g + \sum_{b \in \mc{B}} x_b b$ with $g \in pG$. Let $g' \in G$ be such that $pg' = g$. Then
\[ p(g' + \sum_{b \in \mc{B}} x_b a_b) = g + \sum_{b \in \mc{B}} x_b b.\]
Hence $G \subseteq pG^*$. Given $h \in G^*$, write $h = g + \sum_{b \in \mc{B}} x_b a_b$. Then $ph = pg + \sum_{b \in \mc{B}} x_b b \in G$. So $pG^* \subseteq G$, and so $G = pG^*$.
\end{proof}

If $G$ is a group, recall that we denote by $G[p]$ the elements of $G$ which are torsion of order $p$.

\begin{lemma}
$\mathfrak{H}(G,m)[p] \; / \; (p\mathfrak{H}(G,m))[p] \cong (\mathbb{Z}_p)^{m}$.
\end{lemma}
\begin{proof}
Note that
\begin{equation*}
\begin{split}
\mathfrak{H}(G,m)[p] / (p\mathfrak{H}(G,m))[p]
&\cong \big( G^*[p] / (pG^*)[p] \big) \oplus \big((\mathbb{Z}_p)^m[p] / (p(\mathbb{Z}_p)^m)[p]\big)\\
&\cong (G^*[p] / G[p]) \oplus (\mathbb{Z}_p)^m.
\end{split}
\end{equation*}
We will show that $(G^*[p] / G[p])$ is the trivial group by showing that if $g \in G^*$, $pg = 0$, then $g \in G$. Indeed, write $g = g' + \sum_{b \in \mc{B}} y_b a_b$ with $g' \in G$. Then
\[ 0 = pg = pg' + \sum_{b \in \mc{B}} p y_b a_b = pg' + \sum_{b \in \mc{B}} y_b b.\]
Since $0 \in pG$ has a unique representation (by Lemma \ref{lem:uniq-g}) $0 = 0 + \sum_{b \in \mc{B}} 0b$, we get that $y_b = 0$ for each $b \in \mc{B}$, and so $g = g' \in G$.
\end{proof}

By the previous lemma, we can recover $m$ from $\mathfrak{H}(G,m)$. We have
\[ p\mathfrak{H}(G,m) = pG^* \oplus p (\mathbb{Z}_p)^m \cong pG^* = G \]
so that we can also recover $G$.

Thus, using Lemma \ref{lem:get-m}, $\mc{M} \mapsto \mathfrak{H}(\mathfrak{G}(\mc{M}),\#\mc{M})$ gives a Borel reduction from $T_p$ to abelian $p$-groups.
This completes the proof of Theorem \ref{thm:borel}.

\bibliography{References}
\bibliographystyle{alpha}

\end{document}